\documentclass[12pt]{amsart}
\usepackage[margin=1in]{geometry}
\usepackage{booktabs,bookmark,graphicx,mathtools,microtype}
\usepackage[group-separator={,}]{siunitx}
\usepackage{hyperref} % links and references

\DeclareMathOperator{\V}{Vert} % vertex set
\DeclareMathOperator{\GL}{GL} % General linear group
\DeclareMathOperator{\SL}{SL} % Special linear group
\DeclareMathOperator{\PSL}{PSL} % Projective special linear group
\DeclareMathOperator{\PGL}{PGL} % Projective general linear group
\DeclareMathOperator{\Span}{span} % Span
\DeclareMathOperator{\n}{N}     % Norm
\DeclareMathOperator{\Res}{Res}   %Restriction of scalars
\DeclareMathOperator{\Min}{Min} % Minimum
\DeclareMathOperator{\vol}{vol} % volume

\DeclarePairedDelimiterX\set[1]{\lbrace}{\rbrace}{\def\given{\;\delimsize\vert\;}#1}
\DeclarePairedDelimiterX\abs[1]{\lvert}{\rvert}{\def\given{\;\delimsize\vert\;}#1}
\DeclarePairedDelimiterX\ceiling[1]{\lceil}{\rceil}{\def\given{\;\delimsize\vert\;}#1}

\newcommand{\Hy}{\mathbb{H}} % Hyperbolic space
\newcommand{\Nperf}{N_{\text{perf}}}

\newcommand{\bG}{\mathbf{G}} % Algebraic group
\newcommand{\CC}{\mathbb{C}} % Complexes
\newcommand{\RR}{\mathbb{R}} % Reals
\newcommand{\QQ}{\mathbb{Q}} % Rationals
\newcommand{\ZZ}{\mathbb{Z}} % Integers
\newcommand{\PP}{\mathbb{P}} % Projective space
\newcommand{\OO}{\mathcal{O}} % Ring of integers
\newcommand{\mat}[1]{\begin{bmatrix*} #1 \end{bmatrix*}}
\newcommand{\fn}{\mathfrak{n}}
\newcommand{\fp}{\mathfrak{p}}
\newcommand{\dbar}{\bar{d}}

\newtheorem{theorem}{Theorem}[section]
\newtheorem{lemma}[theorem]{Lemma}
\newtheorem{proposition}[theorem]{Proposition}

\theoremstyle{definition}
\newtheorem{definition}[theorem]{Definition}

\theoremstyle{remark}
\newtheorem{remark}[theorem]{Remark}

\begin{document}

\title{Perfect forms over imaginary quadratic fields}

%\thanks{Last updated: \today}

\author[Scheckelhoff]{Kristen Scheckelhoff}
\address{Department of Mathematics and Statistics\\ 
University of North Carolina at Greensboro\\Greensboro, NC 27412}
\email{kmscheck@uncg.edu}
\urladdr{\url{https://mathstats.uncg.edu/people/directory/kristen-scheckelhoff/}}
\author[Thalagoda]{Kalani Thalagoda}
\address{Department of Mathematics and Statistics\\ 
University of North Carolina at Greensboro\\Greensboro, NC 27412}
\email{kmthalag@uncg.edu}
\urladdr{\url{https://mathstats.uncg.edu/people/directory/kalani-thalagoda/}}
\author[Yasaki]{Dan Yasaki}
\address{Department of Mathematics and Statistics\\ 
University of North Carolina at Greensboro\\Greensboro, NC 27412}
\email{d\_yasaki@uncg.edu}
\urladdr{\url{https://mathstats.uncg.edu/yasaki/}}


\begin{abstract}
In this work, we compute the perfect forms for all imaginary quadratic fields of absolute discriminant up to $5000$ and study the number and types of the polytopes that arise.  We prove a bound on the combinatorial types of polytopes that can arise regardless of discriminant, and give a volumetric argument for a lower bound on the number of perfect forms as well as a heuristic for a better lower bound for imaginary quadratic fields of sufficiently large absolute discriminant.
\end{abstract}

\maketitle

%\linenumbers
\section{Introduction}
Quadratic forms are central objects in mathematics that have been studied for centuries. The geometric theory of quadratic forms, developed by Hermite, can be viewed as part of Minkowski's geometry of numbers. The Hermite problem about finding the arithmetical minima of positive definite quadratic forms is equivalent to the problem of densest lattice packings \cite{schurmann-book}. 

One approach to studying quadratic forms is using
Voronoi's reduction theory \cite{VoronoiI} of perfect quadratic forms.
Voronoi proves there is an infinite polyhedron $\Pi$ in the space of
quadratic forms on which the arithmetic group $\Gamma = \GL_n(\ZZ)$
acts.   The polyhedron is constructed using the minimal vectors of
quadratic forms.  Specifically, the faces of $\Pi$ determine the
possible configurations of minimal vectors of quadratic forms.  The
structure of $\Pi$ captures much of the arithmetic information of
$\Gamma$.   

Indeed, in the classical case of $\Gamma = \SL_2(\ZZ)$,
$\Pi$ descends modulo scaling to give the Farey tessellation of $\Hy$ shown in Figure~\ref{fig:tess-and-spine}.  The tessellation gives rise to a complex which can be used to compute effectively with classical holomorphic modular forms. The main approach is the modular symbol method, introduced by Birch \cite{birch} in 1971 and formalized by Manin \cite{manin} and further developed by Mazur \cite{mazur} and Cremona \cite{cremona-mod}.
The technique for computing the action of the Hecke algebra on modular forms has an
interpretation in this setting in terms of the edges of this
tessellation.  Homothety classes of \emph{perfect forms}, forms that
are uniquely determined by their minimal vectors and arithmetic
minimum, are in bijection with the facets of $\Pi$.  In this setting,
each ideal triangle in the Farey tessellation of $\Hy$ is
$\Gamma$-equivalent to the ideal triangle with vertices $\set{\infty,
  0, 1}$.  This implies that every perfect binary quadratic form is
$\Gamma$-equivalent to the form $\phi(x,y) = x^2 - xy + y^2$.  
Dual to the tessellation is the trivalent tree of homothety classes of
\emph{well-rounded binary 
quadratic forms}, forms whose minimal vectors span $\RR^2$.  
Bass-Serre theory \cite{serre} allows one to use the action of $\Gamma$ on this
tree to recover the amalgam structure of $\Gamma$.  For groups acting on complexes of higher dimensions,  there is an analogous theory due to Gersten and Stallings \cite{Sta91}, Haefliger \cite{Hae91}, and 
Corson \cite{Cor92}. 

\begin{figure}
    \centering
    \includegraphics[trim=12 0 12 0, clip]{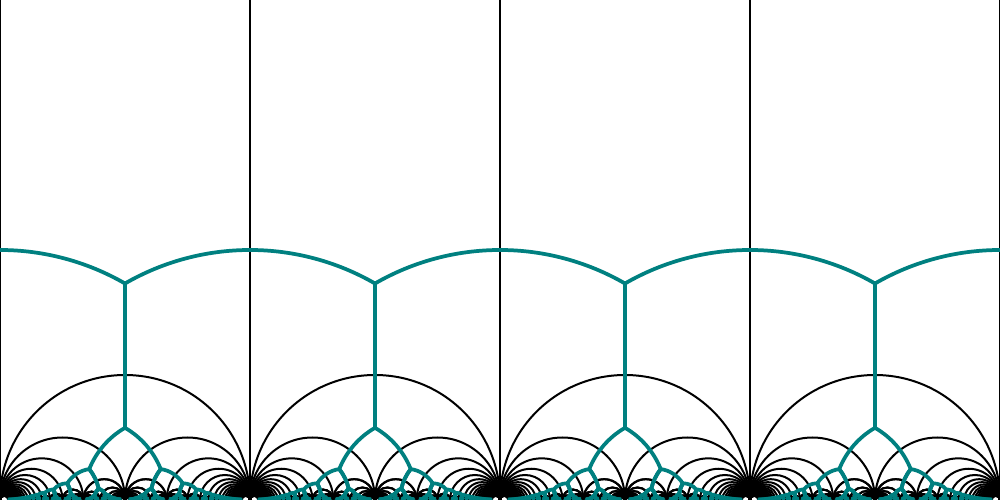}
    \caption{Triangulation of the Poincar\'e upper half-plane by ideal triangles shown here in black.  The well-rounded retract is the  trivalent tree shown in teal.}
    \label{fig:tess-and-spine}
\end{figure}

These concepts have important analogues for $\Gamma = \GL_n(\OO_F)$,
where $\OO_F$ is the ring of integers in a number field $F$.  This project
focuses on the case $n = 2$ with $F$ an imaginary quadratic field.  In Section~\ref{sec:voronoi}, we briefly outline the general Voronoi-Koecher theory of perfect forms and give a bit of the history.  In Section~\ref{sec:imquad}, we specialize to the imaginary quadratic field case and set notation.  
In Section~\ref{sec:results}, we give our main results.  First, we extend the explicit computations of Cremona \cite{Cremona1984} and of the last author in \cite{yasaki-bianchi} classifying perfect Hermitian forms to include all imaginary quadratic fields of absolute discriminant up to $5000$.  The results are shown in Fig.~\ref{fig:poly-total}.  In Theorem~\ref{thm:minbound}, we give a bound on the combinatorial complexity of the configuration of minimal vectors of perfect binary Hermitian forms over imaginary quadratic fields, independent of the discriminant of the field.  In Theorem~\ref{thm:bound}, we give a lower bound on the number of perfect binary Hermitian forms over imaginary quadratic fields, as a function of the discriminant of the field.   

We thank the referee for their helpful comments. The second author thanks the UNCG Graduate School for their support through the Summer Research Fellowship, when part of this research was conducted.

\section{Voronoi complex and perfect forms} \label{sec:voronoi}
Let $F$ be a number field with ring of
integers $\OO_F$. The space of positive definite 
quadratic/Hermitian forms over $F$ form an open cone in a real vector
space. There is a natural decomposition of this cone 
into polyhedral cones corresponding to the facets of the Voronoi-Koecher
polyhedron $\Pi$ \cite{koecher,ash77}.  

The Voronoi complex is the result of a 
polyhedral reduction theory for $\Gamma$ developed by Ash
\cite[Ch.~II]{AMRT_SmoothCompactification} and Koecher \cite{koecher},
generalizing Voronoi's work \cite{VoronoiI} on perfect quadratic forms 
over $\QQ$.  

  Let $X$ be the symmetric space associated to the group of real points of the restriction of scalars of the general linear group  $\bG =\Res_{F/\QQ}(\GL_n)$.  For $\Gamma =\GL_n(\OO_F)$, their work asserts the existence of a $\Gamma$-invariant
tessellation of $X$ coming from a polyhedral cone decomposition of a
space of positive definite quadratic/Hermitian forms.  A spine for $\Gamma$, a
$\Gamma$-equivariant deformation retract of $X$ of dimension equal to
the virtual cohomological dimension of $\Gamma$, is guaranteed by
this work and that of Ash-McConnell \cite{AM-wellrounded} and can be recovered
from the Voronoi tessellation.  The CW-structure of
the tessellation and associated spine can be used to compute the
cohomology of $\Gamma$, and the action of Hecke operators on the
complex group cohomology can also be
described in terms of this structure.

Perfect forms are quadratic/Hermitian forms which are uniquely defined
by their minimum and set of minimal vectors. 
Rational perfect forms have been well-studied, and with current
computing power, $n$-ary forms have been 
classified for $n \leq 8$ \cite{perf8}. See Table~\ref{tab:Qperf} for
the details.

The situation for perfect forms
over number fields is far from complete. Indeed, there 
is even disagreement over what perfection should mean for these
forms.  This project uses perfection in the sense of Koecher and Ash
 described below.

Koecher \cite{koecher} and Ash \cite{ash77,
  AMRT_SmoothCompactification} give a geometric definition of 
\emph{perfect forms}. They generalize the Voronoi polyhedron 
$\Pi$ and define a quadratic form $\phi$ to be \emph{perfect} if the
minimal vectors of $\phi$ are the vertices 
of a facet of $\Pi$. In \cite{Cou} Coulangeon introduces a different
notion of perfection for quadratic forms over 
$F$ known as \emph{Humbert forms}. Coulangeon and Watanabe \cite{CW}
extend Voronoi's theorem on 
extreme forms to Humbert forms; namely, a form is extreme if and only
if it is perfect and eutactic. 

The top-dimensional cells in the Voronoi complex are in
bijection with $\Gamma$-equivalence classes of perfect $n$-ary
forms.  Thus a first step in computing the Voronoi complex is
enumerating such forms.  The structure of such forms (number of
$\Gamma$-equivalence classes of forms, configurations of their minimal
vectors) as $n$ and $F$ vary is not well understood outside of
explicit computations.   Many such explicit computations have been
carried out, and a survey of some of the known results is given below.

Over $F = \QQ$, the number $\Nperf$ of $\GL_n(\ZZ)$-equivalence
classes of perfect forms appears to 
grow rapidly with $n$ (Table~\ref{tab:Qperf}).  
\begin{table}
\renewcommand{\arraystretch}{1.1}
  \caption{The number $\Nperf$ of $\GL_n(\ZZ)$ equivalence classes of
    perfect forms.}\label{tab:Qperf}
  \begin{tabular}{ccc}
\toprule
$n$ &  $\Nperf$ & \text{Authors}\\
\midrule
2&1& Voronoi \cite{VoronoiI}\\
3&1& Ibid.\\
4&2& Ibid.\\
5&3& Ibid.\\
6&7& Barnes \cite{barnes}\\
7&33& Jacquet-Chiffelle \cite{jaquet, jaquet-chiffelle}\\
8&10916&Dutour-Schurmann-Vallentine \cite{perf8}\\
\bottomrule   
\end{tabular}
\end{table}

Computation of the Voronoi complex  requires not just the perfect forms, but information about lower-dimensional cells as well.  As the number and complexity of the top-dimensional cells becomes large, computation of the lower-dimensional cells becomes infeasible.  This bottleneck has been 
overcome for $n \leq 7$ in work of Elbaz-Vincent-Gangl-Soul\'e
\cite{PerfFormModGrp} by taking advantage of 
external symmetries of each cell.  The $n = 8$ case is still out of
reach, primarily due to the complexity of the $E_8$ lattice.  The
shortest vectors of this lattice form a convex polytope with 
\num{25075566937584} facets.  

Over other number fields, the data is far less
complete.  The computations, even for $n = 1$, are nontrivial.  In
Table~\ref{tab:pf-nf}, the known classifications of $n$-ary
perfect forms over number fields are listed.

\begin{table} 
\renewcommand{\arraystretch}{1.1}
  \caption{Classification of $\GL_n(\OO_F)$-equivalence classes of
    perfect $n$-ary forms over number field $F$.} \label{tab:pf-nf}
\begin{tabular}{ccc}
\toprule
  $n$ & Field & Authors\\
\midrule
1 & $\QQ(\sqrt{d})$, $0 < d \leq 200000$, $d$ squarefree & Y. \cite{Yasunary}\\
& $\QQ(\zeta_m)$, $\phi(m) \leq 17$ & Sigrist \cite{sigrist}\\
& $\QQ(\alpha)$, $\alpha^3 + \ell \alpha - 1 = 0$,
$4\ell^3 + 27$ squarefree & Komatsu-Watanabe \cite{komatsu-watanabe}\\
\midrule
2 & $\QQ(\sqrt{d})$, $d \in \set{2,3,5,6}$ & Ong \cite{Ong}, Leibak \cite{lei6}\\
  & $\QQ(\sqrt{-d})$, $0 < d \leq 100$, $d$ squarefree & Cremona \cite{Cremona1984}, Y. \cite{yasaki-bianchi}\\
 & $\QQ(\zeta_5)$& Y. \cite{Yascyclotomic}\\
& $\QQ(\alpha)$, $\alpha^3 - \alpha^2 + 1 = 0$ &
Gunnells-Y. \cite{gy_cubic}\\
 & $\QQ(\zeta_{12})$& Jones \cite{jones}\\
  & $\QQ(\zeta_{8})$ and $\QQ(\alpha)$, $\alpha^4-\alpha^3+2\alpha^2+x+1=0$ & Jones, Sengun \cite{modp}\\
\midrule
3 & $\QQ(\sqrt{-1})$ & Staffeldt \cite{staffeldt}\\
 &  $\QQ(\sqrt{D})$, discriminant $-24 \leq D < 0$ & AIM group \cite{imquad-coh}\\
\midrule
4 &  $\QQ(\sqrt{D})$, $D \in \set{-3, -4}$& AIM group \cite{imquad-coh}\\
\bottomrule
\end{tabular}
\end{table}

\section{Perfect forms over imaginary quadratic fields} \label{sec:imquad}
In this section, we give just enough details to set the relevant notation.  We follow \cite[\S 3]{yasaki-bianchi} and \cite[\S 2 and \S 6]{imquad-coh} closely, and the reader should reference these for additional details and a description of the algorithms involved.

Fix a square-free positive integer $d$.  Let $F$ be the imaginary quadratic field $F=\QQ(\sqrt{-d})$, with ring of algebraic integers $\OO_F$.  Then $F$ has discriminant $\Delta = -4d$ if $d \equiv 1,2\bmod 4$, and $\Delta = -d$ otherwise.  The ring of integers $\OO_F$ is equal to $\ZZ[\omega]$, where
\[
\omega = \begin{cases}
\sqrt{-d} & \text{if $d \equiv 1,2\bmod 4$}\\
\frac{1+\sqrt{-d}}{2} & \text{if $d \equiv 3\bmod 4$}.
\end{cases}
\]

We fix a complex embedding $F \hookrightarrow \CC$ and
identify $F$ with its image. We extend this identification to
vectors and matrices as well, and use $\bar{\cdot}$ to denote
complex conjugation on $\CC$, the non-trivial Galois automorphism on $F$.

Let $V$ be the $4$-dimensional real vector space of $2 \times 2$ complex Hermitian matrices with complex coefficients,
\[
V=\set*{ \mat{
a&b\\\bar{b}&c}
 \given a,c \in \RR, b \in \CC }.
\]
Let $C \subset V$ denote the subset of positive definite matrices.  Then $C$ is a codimension $0$ open cone. Here the boundary of the cone consists of semi-definite Hermitian forms; below we see that minimal vectors of a Hermitian form can be represented by elements on the boundary of $C$.

Using the chosen complex embedding of $F$, we view $V_F$, the $2 \times 2$ Hermitian matrices with entries in $F$, as a subset of $V$.  Define a map $q \colon \OO_F^2 \setminus
\set{0} \to V$ by $q(x) = x \bar{x}^t$.  For each $x \in \OO_F^2$, we have that $q(x)$ is on the boundary of $C$.  Let $C^*$ denote the union of $C$ and the image of $q$.  

The group $\GL_2(\CC)$ acts on $V$ by $g \cdot A = g A \bar{g}^t$.   The image of $C$ in the quotient of $V$ by positive homotheties can be identified with hyperbolic $3$-space $\Hy^3$.  The image of $q$ in this quotient is identified with $\PP^1(F) = F \cup \set{\infty}$, the set of cusps.

Each $A \in V$ defines a Hermitian form $A[x] = \bar{x}^tAx$, for $x \in \CC^2$.  Using the chosen complex embedding of $F$, we can view $\OO_F^2$ as a subset of $\CC^2$.

\begin{definition} \label{def:minimum}
  For $A \in C$, we define the \emph{minimum of $A$} as 
\[\min(A) \coloneqq  \inf_{v \in \OO_F^2 \setminus \set{0}} A[v].\]
Note that $\min(A)>0$ since $A$ is positive definite.
A vector $v \in \OO_F^2$ is called a
\emph{minimal vector of $A$} if $A[v] = \min(A)$.
We let $\Min(A)$ denote the set of minimal vectors of $A$.
\end{definition}

Since $\OO_F^2$ is discrete in the topology of $\CC^2$,
the set $\Min(A) $ is finite.  
In fact, in Theorem~\ref{thm:minbound}, we give an upper bound on the number of minimal vectors for any positive definite Hermitian form over an imaginary quadratic field, independent of the field.
A minimal vector $\mat{\alpha\\ \beta} \in \OO_F^2$ generates an ideal $(\alpha, \beta) \subseteq \OO_F$ that has minimal norm among ideals in its class in the class group of $F$.

\begin{definition}\label{def:perfect}
We say a Hermitian form $A \in C$ is a \emph{perfect Hermitian 
form over $F$} if
\[\Span_\RR\set{q(v) \given v \in \Min(A)} = V.\]
\end{definition}

The above definition is equivalent to the more classical definition of a perfect form, as a Hermitian form that is completely determined by its set of minimal vectors.  That is, a Hermitian form $A$ is \emph{perfect} if $\Min(A)$ determines $A$ up to scaling by $\RR^{+}$.

\begin{definition} A \emph{polyhedral cone} in $V$ is a subset $\sigma$ of the form 
\[\sigma = \set*{\sum_{i = 1}^n \lambda_i q(v_i) \given \lambda_i \geq
  0},\]
where $v_1, v_2, \dots, v_n$ are non-zero vectors in $\OO_F^2$.  
\end{definition}

\begin{definition}
A set of polyhedral cones $S$ forms a \emph{fan} if the following two conditions hold:
\begin{enumerate}
\item If $\sigma$ is in $S$ and $\tau$ is a face of $ \sigma $, then
  $\tau$ is in $S$.
\item If $\sigma$ and $\sigma'$ are in $S$, then $\sigma \cap \sigma'$ is a
  common face of $\sigma$ and $\sigma'$.
\end{enumerate}
Note that a face here can be of codimension higher than~$1$.
\end{definition}

\begin{theorem}\label{thm:koecher}
  There is a fan $S$ in $V$ with $\GL_2(\OO_F)$-action such
  that the following hold.
  \begin{enumerate}
  \item There are only finitely many $\GL_2(\OO_F)$-orbits in
    $S$.
  \item Every $y \in C$ is contained in the interior of a unique cone
    in $S$.
  \item Any cone $\sigma \in S$ with non-trivial
    intersection with $C$ has finite stabilizer in
    $\GL_2(\OO_F)$. \label{it:finite-stab} 
  \end{enumerate}
The $4$-dimensional cones in $S$ are in bijection with
perfect forms over $F$.
\end{theorem}

The bijection is explicit and allows one to compute the structure of
$S$ using a modification of Voronoi's algorithm
\cite[\S 2, \S 6]{imquad-coh}.  Specifically, $\sigma$  is a
$4$-dimensional cone in $S$ if and only if 
there exists a perfect Hermitian form $A$ such that
\[\sigma = \set*{\sum_{v \in \Min(A)} \lambda_v q(v) \given \lambda_v
  \geq 0}.\]
Modulo positive homotheties, the fan $S$ descends to a $\GL_2(\OO_F)$-tessellation of $\Hy^3$ by ideal polytopes. 

\begin{remark} \label{rem:min}
If $v \in \Min(A)$, then $-v \in \Min(A)$.  Following the conventions for the classical case, we pick exactly one representative in $\set{\pm v}$ to include in $\Min(A)$.
Furthermore, the value $A[v]$ only depends on the image $q(v)$.  When the class number of $F$ is greater than $1$, there may be additional vectors with the same image.  For example, let $F$ be the imaginary quadratic field of discriminant $-91$, and let $\OO_F = \ZZ[\omega]$, where $\omega = \frac{1 + \sqrt{-91}}{2}$.  Then $\pm\mat{\omega  + 1\\
-\omega + 4}$ and $\pm \mat{5\\-\omega - 3}$ both have the same image under $q$.  They give rise to the same cusp $\frac{\omega  + 1}{-\omega + 4} = \frac{5}{-\omega - 3} = \frac{w - 4}{7}$.  The prime $5$ splits in $\OO_F$, so that $5\OO_F = \fp_5\bar{\fp}_5$, and the ideal $(\omega  + 1,-\omega + 4)$ is $\fp_5$, while the ideal $(5,-\omega - 3)$ is $\bar{\fp}_5$.
It follows that an ideal polytope in $\Hy^3$ determined by a perfect form over $F$ may have strictly fewer vertices than the form has minimal vectors.
\end{remark}

\section{Results}\label{sec:results}
We computed the perfect forms for all imaginary quadratic fields $F$ of absolute discriminant up to $5000$, extending previous explicit computations \cite{Cremona1984} and \cite{yasaki-bianchi}. In \cite{yasaki-bianchi}, the third author computed perfect forms for all the class number 1 and 2 cases and for all the imaginary quadratic fields $\QQ(\sqrt{-d})$ with square free $d < 100$, which was a total of 69 fields. We expanded this computation to compute all fields with discriminant less than $5000$ which was a total of $ 1524$ fields. See Figure~\ref{fig:poly-total} for a plot of $\Nperf(F)$ as a function of the discriminant of $F$.

  \begin{figure}
    \centering
    \includegraphics[width=0.75\textwidth]{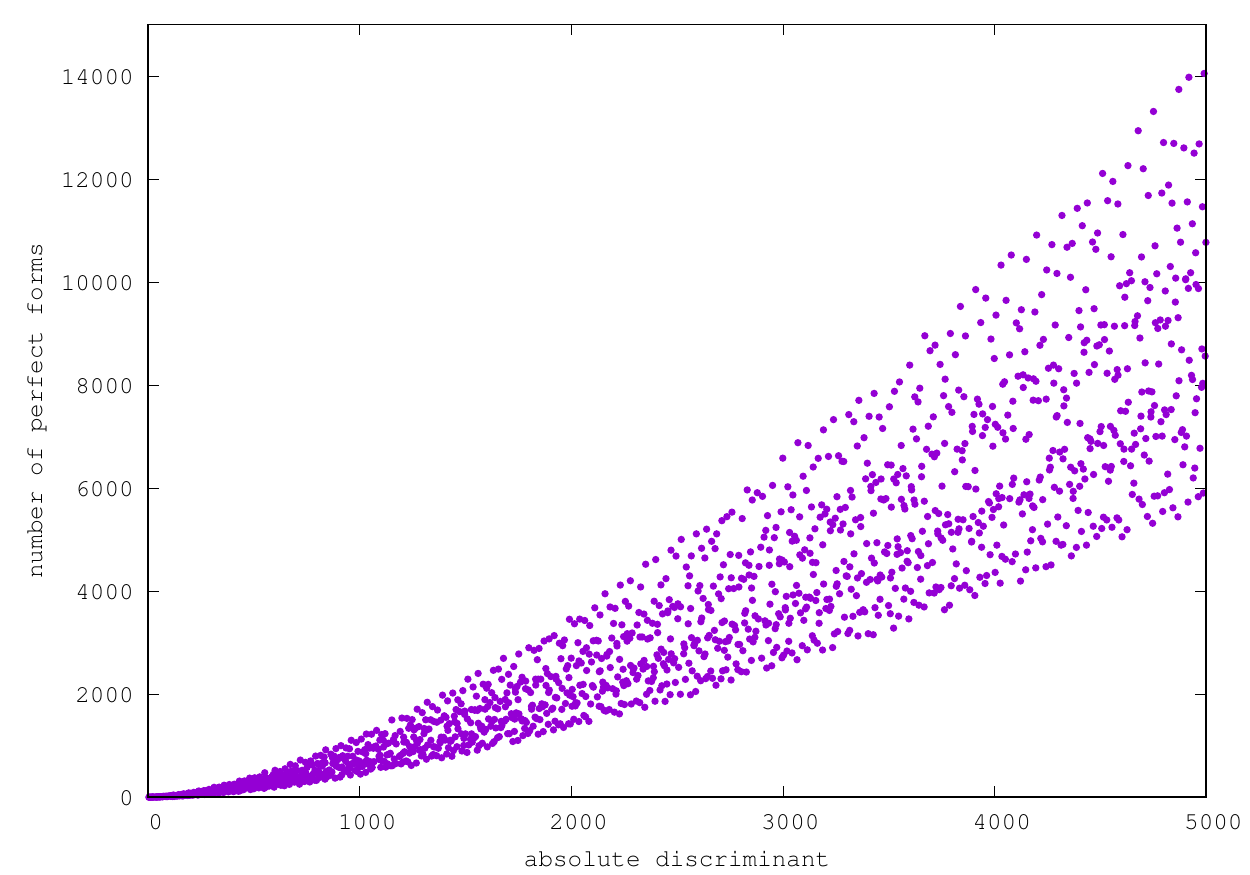} 
    \caption{Number of perfect forms $\Nperf(F)$, indexed by absolute discriminant of $F$.}
    \label{fig:poly-total}
\end{figure}

The following theorem ensures that the configurations of minimal vectors of perfect binary Hermitian forms do not get arbitrarily complicated.  In particular, the number of minimal vectors is bounded, independent of the imaginary quadratic field, so there are only a finite number of combinatorial types of ideal polytopes arising in the Voronoi tessellation of $\Hy^3$.

\begin{theorem}\label{thm:minbound}
  Let $A$ be a positive definite binary Hermitian form over an imaginary quadratic field.  Then 
  \[
  \# \Min(A)  \leq 12.
  \]
\end{theorem}

\begin{proof}
Let $V_\QQ$ denote the set of $4 \times 4$ symmetric matrices with entries in $\QQ$, viewed as quaternary forms over $\QQ$.  Recall for each imaginary quadratic field $F$, we let $V_F$ denote the set of $2 \times 2$ Hermitian matrices with entries in $F$, viewed as Hermitian forms over $F$.  We construct a map $\Phi \colon V_F \to V_\QQ$ which takes a positive definite Hermitian form over $F$ to a positive definite quaternary form over $\QQ$.

  There are exactly two perfect forms in the space of positive definite quaternary forms over $\QQ$ \cite{Korkine}; one with $12$ minimal vectors and one with $10$ minimal vectors.  It follows that every positive definite quaternary form over $\QQ$ has at most $12$ minimal vectors.

  Fixing our $\ZZ$-basis $\set{1,\omega}$ for $\OO_F$, we get a bijection $\phi\colon \OO_F^2 \to \ZZ^4$.  This induces a  map $\Phi \colon V_F \to V_\QQ$ such that $\Phi$ preserves vector evaluation;
  \[A[v]=\Phi(A)[\phi(v)], \quad \text{for all $v \in \OO_F^2$}.\]

Note the value of $\min(\Phi(A))$ is determined by $\min(A)$, since vector evaluation is preserved by $\Phi$.  Since there are at most $12$ minimal vectors for any quaternary form  $\Phi(A)$, there can be at most $12$ minimal vectors for $A$.
\end{proof}

The bound in the above theorem is sharp. In particular, there are perfect forms with exactly $12$ minimal vectors, up to sign.  Furthermore, despite the phenomena described in Remark~\ref{rem:min}, where the number of vertices is strictly less than the number of minimal vectors, the bound of $12$ vertices here is sharp.  In particular, there are imaginary quadratic fields $F$ for which there is a perfect form that gives rise to an ideal polytope with $12$ vertices.  

Since the minimal vectors of each perfect form map to vertices of their corresponding polytopes in the cone, and there are finitely many combinatorial types of polytopes with $12$ or fewer vertices, only finitely many types of polytopes can arise in a tessellation of $\Hy^3$ as we vary the discriminant of the imaginary quadratic field.  

\begin{remark}
 According to the data from \cite{polysmall}, there are  \num{6860405} combinatorial types of $3$-dimensional polytopes with at most $12$ vertices. However, in the range of our computational investigation, we only observed $8$ distinct combinatorial types of polytopes: tetrahedron, octahedron, cuboctahedron, triangular prism, hexagonal cap, square pyramid, truncated tetrahedron, and triangular dipyramid. These were also the combinatorial types of polytopes observed in \cite{yasaki-bianchi}. Table \ref{tab:poly-types} summarizes the information about how often each type of polytope was observed and the number of fields which witnessed these in the range of our computation. 
  
\end{remark}
\begin{table}
\label{tab:poly-types}
\centering
\[
\begin{array}{llcc}
   \toprule   
   \text{} &\text{Polytope Type}  & \text{Number of Fields} & \text{Percentage of Polytopes }\\
\midrule
  \includegraphics[scale = 0.06]{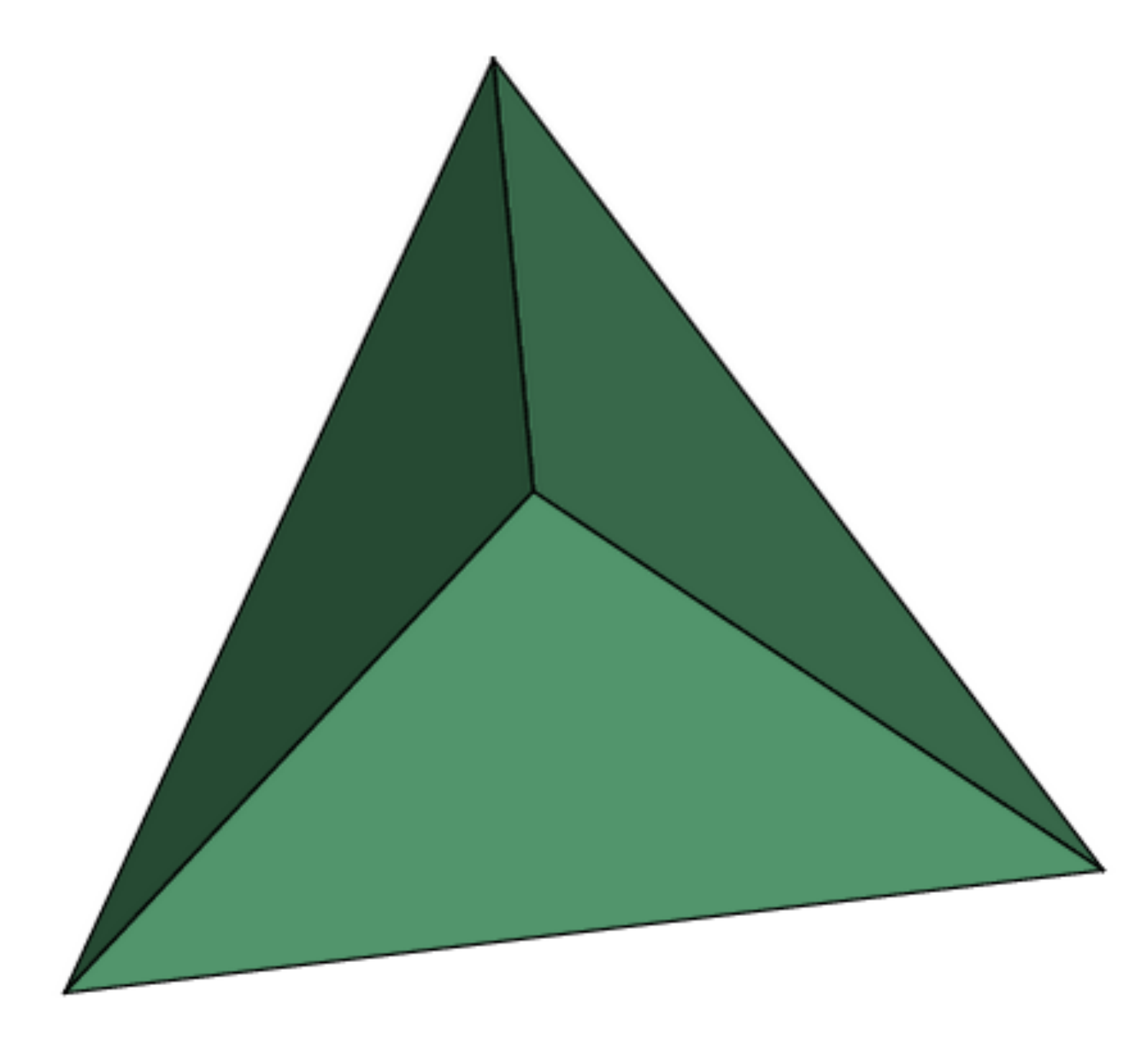}& \text{tetrahedron } &1504& 91.524  \\
  \includegraphics[scale = 0.08]{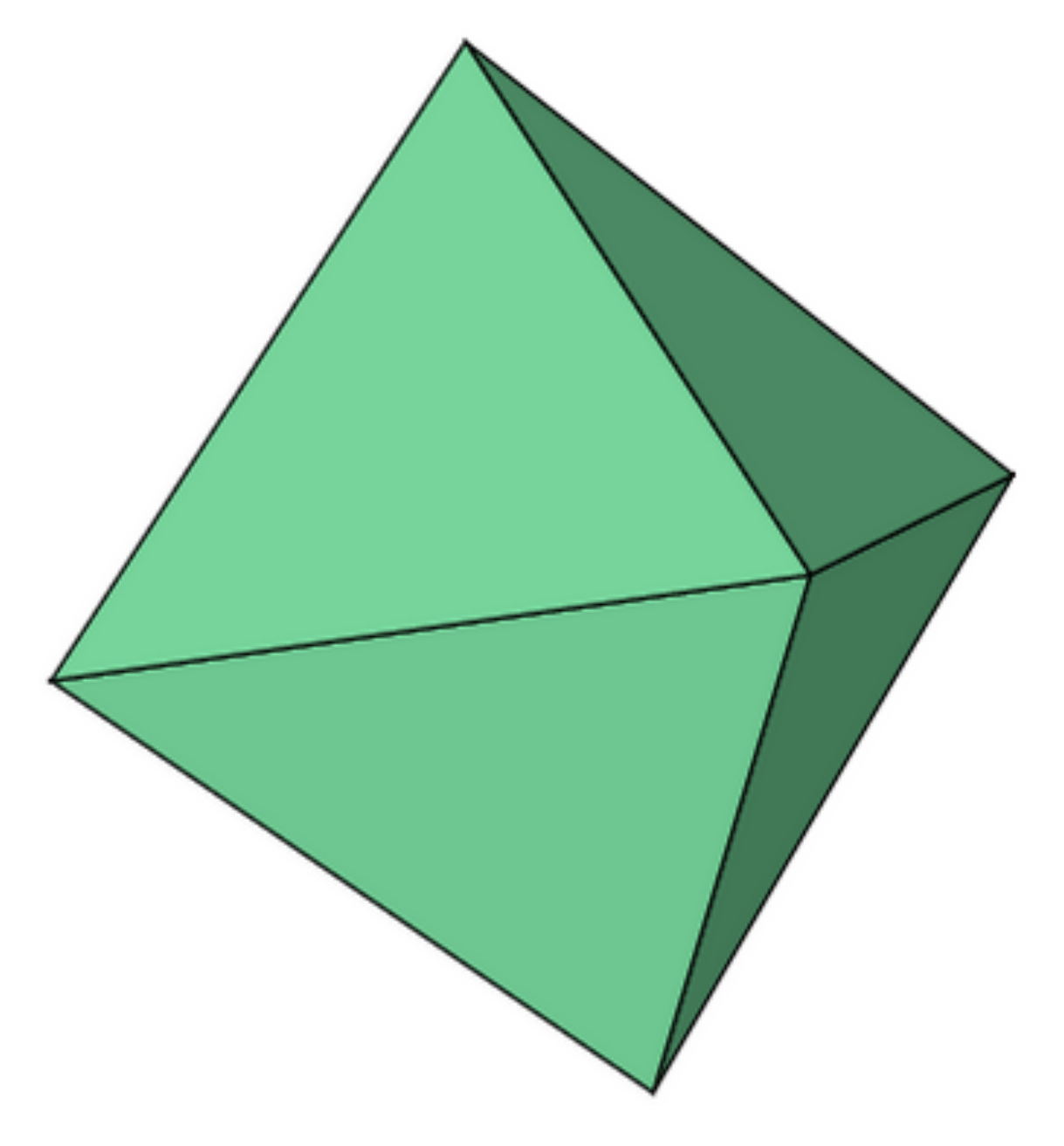}&\text{octahedron } &912 & 0.066 \\
  \includegraphics[scale = 0.05]{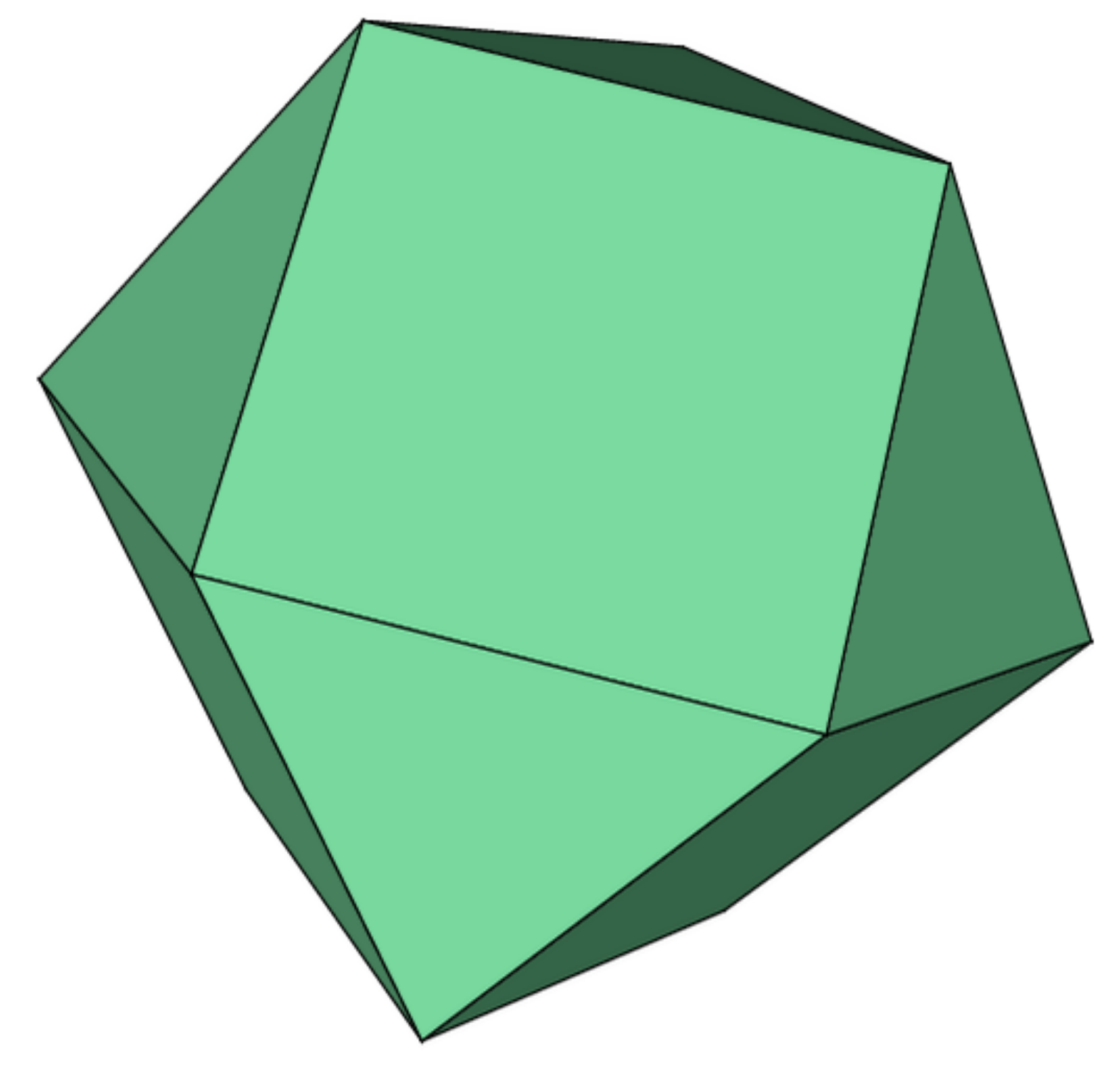}&\text{cuboctahedron } &16& 0.005\\
  \includegraphics[scale = 0.05]{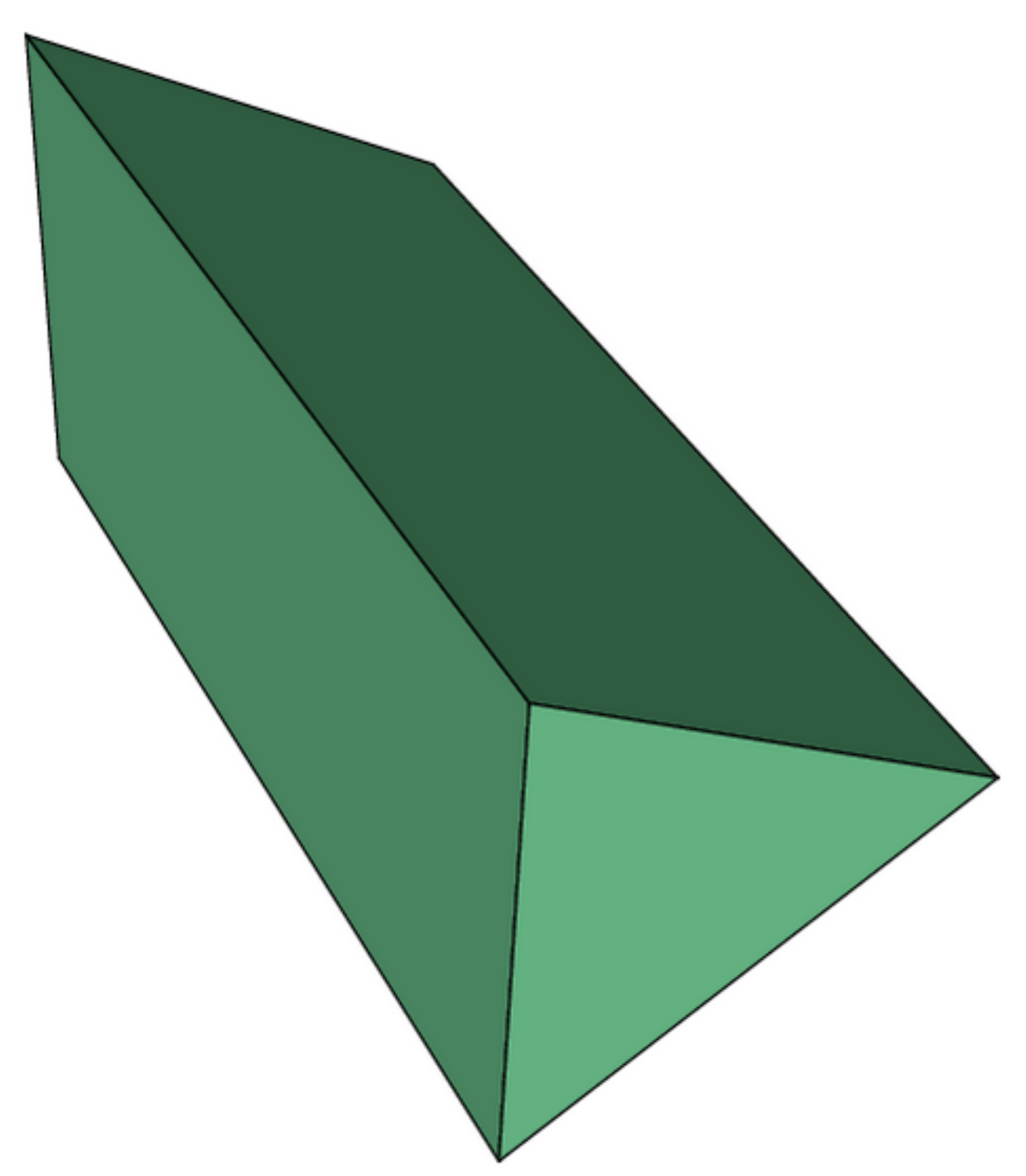}&\text{triangular prism} &1511& 2.416\\
  \includegraphics[scale = 0.07]{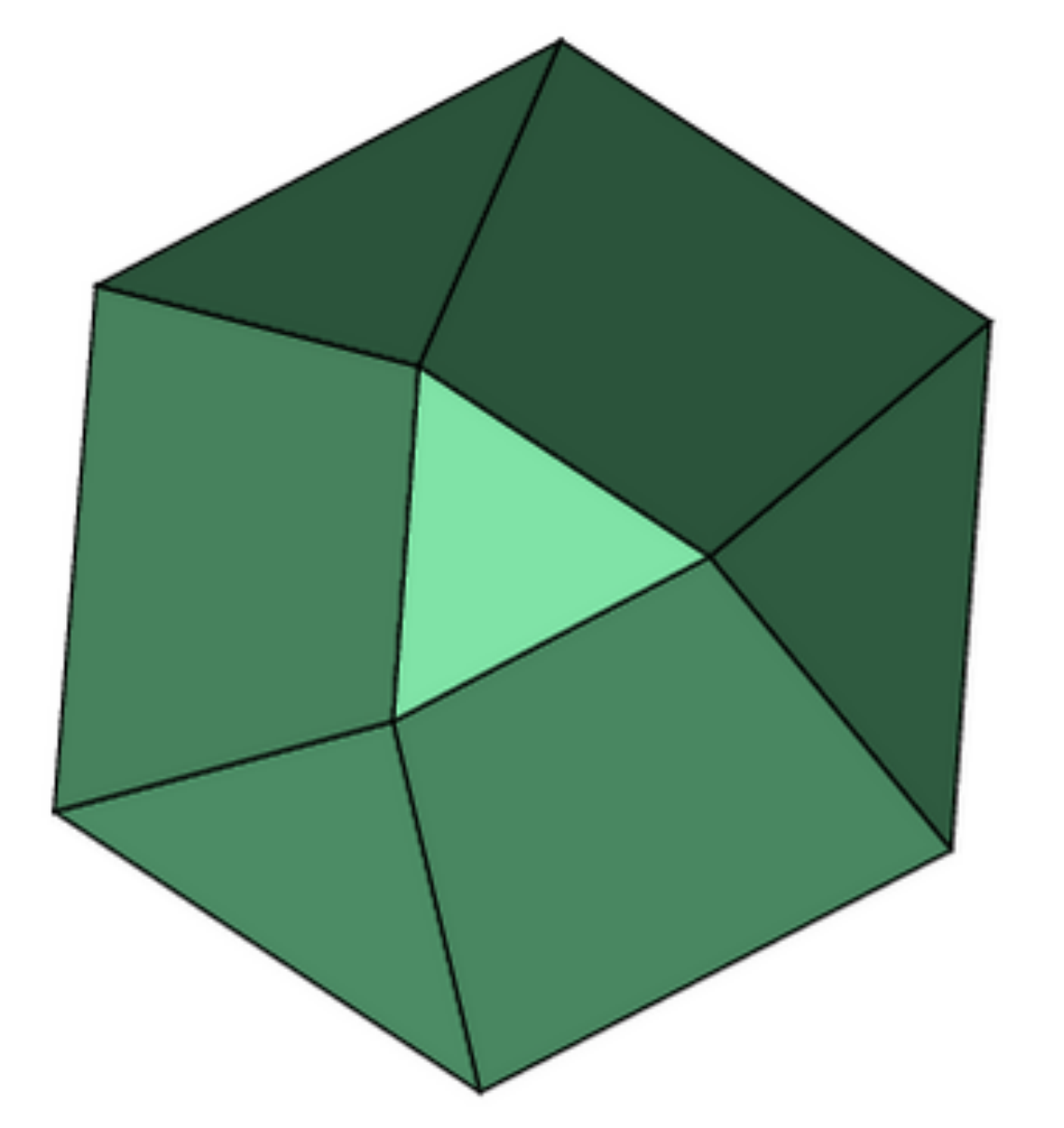}&\text{hexagonal cap }&1358& 0.199 \\
  \includegraphics[scale = 0.05]{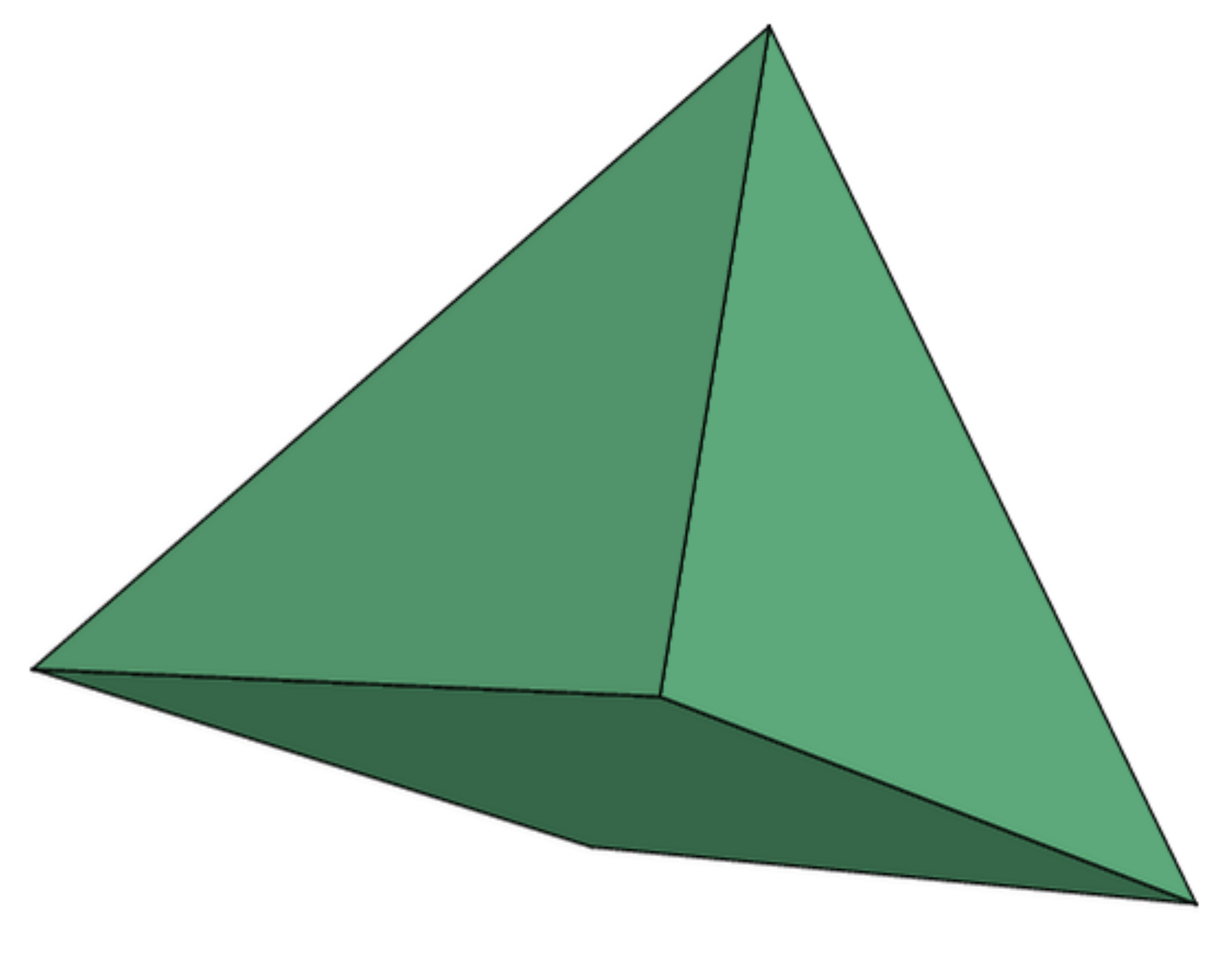}&\text{square pyramid }&1506& 5.764\\
  \includegraphics[scale = 0.1]{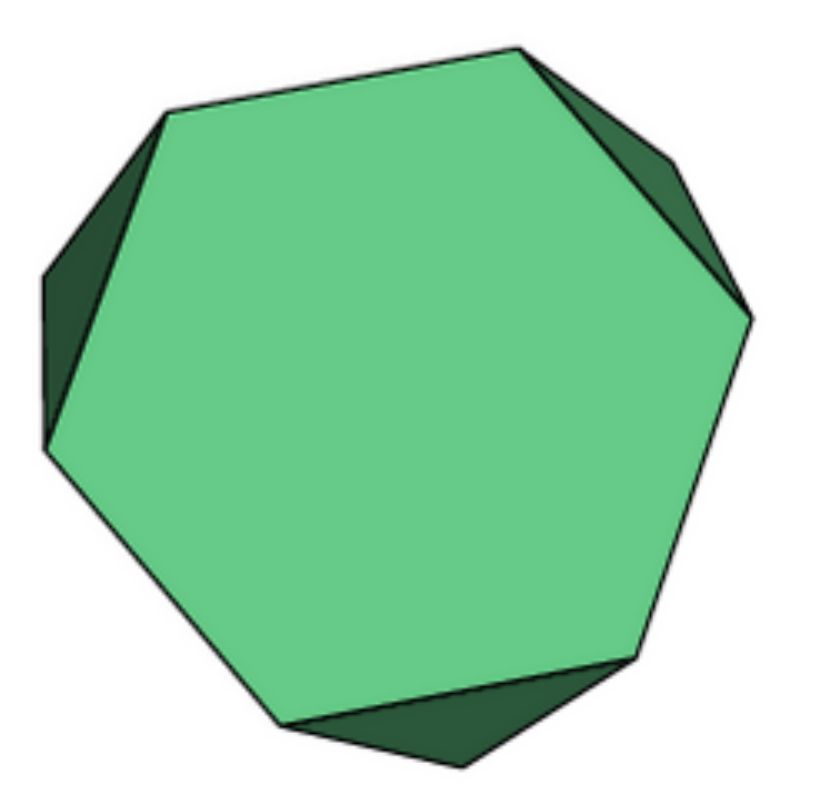}&\text{truncated tetrahedron}&60& 0.007\\
  \includegraphics[scale = 0.1]{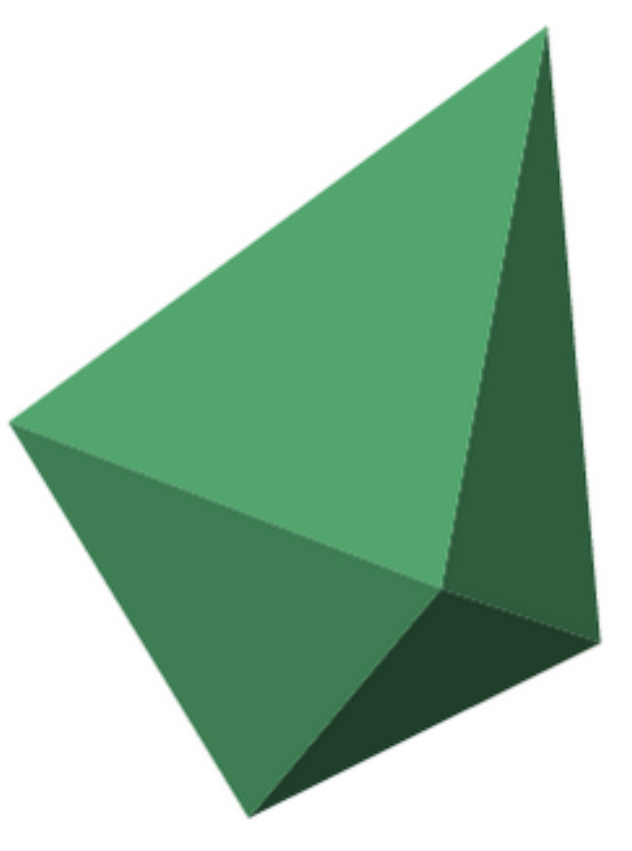}&\text{triangular dipyramid}&416& 0.019\\
  \bottomrule
\end{array}
\]
    \caption{The total number of fields that witnessed each polytope type and the total percentage of each polytope type observed in the range of computation. }
    \label{table:my_poly}
\end{table}

\begin{remark} \label{rem:tetra}
As the discriminant increases, the total number of polytopes increases and appears to be dominated by tetrahedra.  The types of observed polytopes as a percentage of the total number of polytopes computed (up to absolute discriminant $5000$) are plotted in Figure~\ref{fig:poly-pct-all}.  
\end{remark}

\begin{figure}
    \centering
    \includegraphics[width=0.7\textwidth]{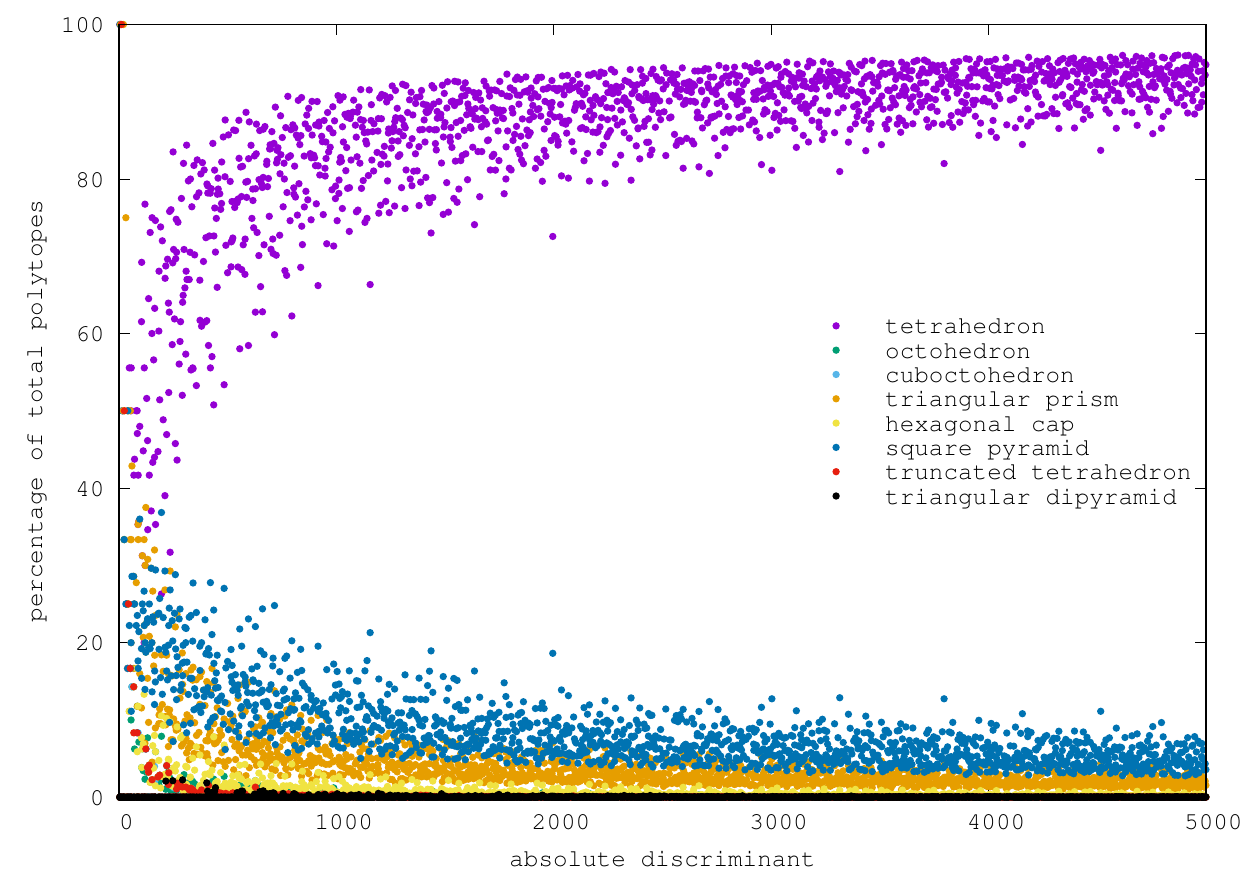}
    \caption{ Observed polytope types as a percentage of total number of polytopes, indexed by absolute discriminant of $F$. }
    \label{fig:poly-pct-all}
\end{figure}
\newpage
Next, we compute a lower bound for the number of perfect forms over $F$ by considering decompositions of the polytopes corresponding to perfect forms in the tessellation of $\Hy^3$.  Each such polytope can be decomposed into ideal tetrahedra without introducing new vertices.  Since $\GL_2(\OO_F)$-translates of such tetrahedra cover $\Hy^3$, an upper bound on the number of tetrahedra required for each polytope together with an upper bound on the volume of an ideal tetrahedron gives a lower bound on the number of perfect forms $\Nperf(F)$ involving the volume of the quotient  $\GL_2(\OO_F)\backslash\Hy^3$.  

The following bound on subdividing a polytope must be well known, but we could not find a reference, so we give the proof here.

\begin{lemma}\label{lemma:decom}
%  Bound number of tetrahedra.
Let $P$ be a convex polytope with $V$ vertices, with $V > 6$.  Then there exists a simplicial subdivision of $P$ consisting of at most $2V - 9$ tetrahedra.
\end{lemma}
\begin{proof}
  First, we describe a well-known technique for subdividing $P$ into tetrahedra.  Triangulate each face of $P$ without adding vertices to get a polytope $Q$, so that vertex set $\V(Q)$ is the same as the vertex set $\V(P)$.  Fix a vertex $v$ in $\V(Q)$.  Let $S_v$ denote the set of triangular faces of $Q$ that do not contain $v$ as a vertex.  Then for each $f \in S_v$, form a tetrahedron $t_f$ using $f$ and $v$.  Then $\bigcup_{f \in S_v} t_f$ is a simplicial subdivision of $P$.  

  Let $V$, $E$, and $F$ be the number of vertices, edges and faces of $Q$.  It suffices to show there is a choice of vertex $v$ such that $\#S_v \leq 2V - 9$.  By construction, $V = \#\V(Q) = \#\V(P)$, and $S_v = F - \deg(v)$, where $\deg(v)$ denotes the degree of vertex $v$. Furthermore, the Euler characteristic is $2$, so $V - E + F = 2$.  Since all the faces are triangles, and each edge is on the boundary of exactly two triangular faces, we have $3F= 2E$.  It follows that $E = 3V - 6$ and $F = 2V - 4$.

  Each edge has exactly two boundary vertices, so the average degree $\dbar$ of a vertex of $Q$ is
  \[\dbar = \frac{1}{\#\V(Q)}\sum_{v \in \V(Q)}  \deg(v) = \frac{2E}{V} = 6 - \frac{12}{V}.\]
  It follows that for $V > 6$, we have $\dbar > 4$.  Since the average degree is greater than $4$, there exists a vertex with degree at least $5$.  Let $v$ be a vertex of $Q$ of maximal degree.  Then
  \[\#S_v = F - \deg(v) \leq F - 5 = 2V - 9,\]
  as desired.
\end{proof}

\begin{remark}
  The argument above can be strengthened to show $\#S_v \leq 2V - 10$ for $V > 12$.
\end{remark}

For a group $\Gamma$ acting properly discontinuously on $\Hy^3$, let $\mu(\Gamma)$ denote the volume of $\Gamma \backslash \Hy^3$.  Let $\zeta_F(s)$ be the Dedekind zeta function of $F$,
  \[\zeta_F(s) = \sum_{\fn \subseteq \OO_F} \frac{1}{\n_{F/\QQ}(\fn)^s}. \]
  A classical result of Humbert asserts that for an imaginary quadratic field $F$ of discriminant $\Delta$, we have
\[
    \mu(\PSL_2(\OO_F)) = \frac{\abs{\Delta}^{3/2}}{4\pi^2} \zeta_F(2).    
\]

  \begin{proposition}[{\cite[Theorem 7.3]{Borel1981}}] \label{prop:vol} Let $F$ be an imaginary quadratic field of discriminant $\Delta$. Then
\[\mu(\GL_2(\OO_F))=\frac{\abs{\Delta}^{3/2}}{8\pi^2} \zeta_F(2).\]
\end{proposition}
\begin{proof}
  Since the center of $\GL_2(\OO_F)$ acts trivially on $\Hy^3$,  we have $\mu(\PSL_2(\OO_F)) = \mu(\SL_2(\OO_F))$ and 
$\mu(\PGL_2(\OO_F)) = \mu(\GL_2(\OO_F))$.  It follows that in order to compute $\mu(\GL_2(\OO_F))$, we need to divide $\mu(\PSL_2(\OO_F))$ by the index $[\PGL_2(\OO_F):\PSL_2(\OO_F)]$.  This index is equal to the index of squares in the units of $\OO_F$ by \cite[Lemma 3.1]{torsion}.  This index is independent of the discriminant and is equal to $2$, so the result follows.
\end{proof}

Volumes of tetrahedra in $\Hy^3$ can be expressed in terms of $\Lambda(\theta)$, the Lobachevsky function given by 
\[\Lambda(\theta) = -\int_0^\theta \log\abs{2 \sin(t)} dt.\]
It is known \cite[Corollary, page 20]{Milnor1982} that the ideal tetrahedron $T$ of maximum volume has volume 
\begin{equation}\label{eq:maxtetra}
  \vol(T) = 3\Lambda (\pi/3) \approx 1.0149416\dots.
\end{equation}
For $F = \QQ(\sqrt{-3})$, we have $\Nperf(F) = 1$, and the ideal polytope  corresponding to the perfect form is an example of such a tetrahedron.
We use \eqref{eq:maxtetra} to bound the number of perfect binary Hermitian forms over imaginary quadratic fields below.

\begin{theorem}\label{thm:bound}
Let $F$ be an imaginary quadratic field of discriminant $\Delta$.  The number $\Nperf(F)$ of perfect binary Hermitian forms over $F$  satisfies the following bound:
\[
\Nperf(F) \geq \ceiling*{\frac{\abs{\Delta}^{3/2}}{360\pi^2\Lambda (\pi/3)}\zeta_F(2)}.
\]
\end{theorem}
\begin{proof}
  Let $R$ denote a set of representatives of the $\GL_2(\OO_F)$-orbits of perfect binary Hermitian forms over $F$, so that $\Nperf(F) = \#R$.  For each perfect form $A \in R$, let $p_A$ denote the corresponding ideal Voronoi polytope in $\Hy^3$ guaranteed by Theorem~\ref{thm:koecher}.  Since $\GL_2(\OO_F)$ acts on $\Hy^3$ by isometries and the translates of $R$ cover $\Hy^3$, the volume of the quotient $\GL_2(\OO_F) \backslash \Hy^3$ is bounded by the sum of the volumes of the polytopes in $R$, so 
  \[\mu(\GL_2(\OO_F)) \leq  \sum_{A \in R} \vol(p_A) \leq \Nperf(F) \max_{A \in R} (\vol(p_A)).\]
  By Theorem~\ref{thm:minbound}, each $A$ in $R$ has at most $12$ minimal vectors.  Thus each $p_A$ is an ideal polytope with at most $12$ vertices.  By Lemma~\ref{lemma:decom}, there is a decomposition of $p_A$ into at most $15$ tetrahedra.  It follows that
  \[\max_{A \in R}(\vol(p_A)) \leq 15 \vol(T) = 45 \Lambda(\pi/3),\]
  by \eqref{eq:maxtetra}.  Then
  \[  \mu(\GL_2(\OO_F)) = \frac{\abs{\Delta}^{3/2}}{8\pi^2} \zeta_F(2) \leq  45 \Nperf(F) \Lambda(\pi/3),\]
and the result follows.
\end{proof}

The proof of Theorem~\ref{thm:bound} is obtained by proving each Voronoi polytope has volume less than or equal to $15\vol(T)$, where $T$ is an ideal hyperbolic tetrahedron $T$ of maximum volume.  The observation in Remark~\ref{rem:tetra} is that as the discriminant increases, most of the Voronoi polytopes that arise are tetrahedra.  It follows that a better estimate for the average volume of a Voronoi polytope is $\vol(T)$.  This suggests that we might get a good estimate of $\Nperf(F)$ by considering $15$ times the bound given in Theorem~\ref{thm:bound}.  Let $E(F)$ denote this value,
\[E(F) = \ceiling*{\frac{\abs{\Delta}^{3/2}}{24\pi^2\Lambda (\pi/3)}\zeta_F(2)}.\]
Then indeed the data bears out this idea.  Figure~\ref{fig:estimate-ratio} shows the ratio $\Nperf(F)/E(F)$.  This ratio  is close to $1$ in the range of the computation, so $E(F)$ is a good estimate for $\Nperf(F)$.  As the discriminant increases, the ratio appears to approach a value strictly greater than $1$.  This suggests that while $E(F)$ is not a lower bound for $\Nperf(F)$ for small discriminants, there exists $D$ such that for $\abs{\Delta} \geq D$, we have $\Nperf > E(F)$.

\begin{figure}[h!]
    \centering
    \includegraphics[width=0.75\textwidth]{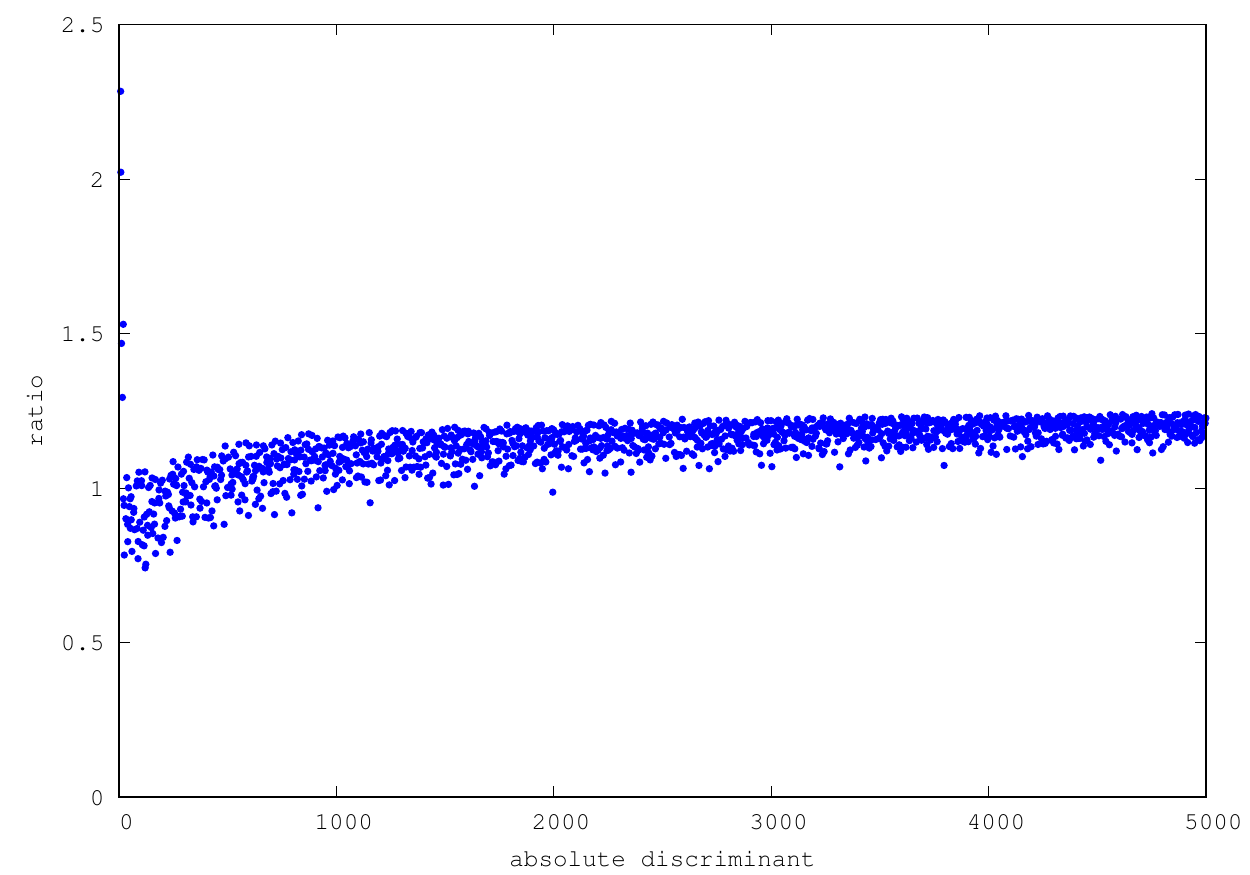} 
    \caption{The ratio $\Nperf(F)/E(F)$ of the number of perfect forms $\Nperf(F)$ to the estimate $E(F)$, indexed by absolute discriminant of $F$. }
    \label{fig:estimate-ratio}
\end{figure}

\bibliographystyle{plain}
\bibliography{imquad-hermitian-forms}

\end{document}